\documentclass[article, 11pt]{article}
\usepackage[utf8]{inputenc}
\usepackage{amsmath, amssymb, amsfonts, amsthm, mathtools, hyperref, thm-restate}
\usepackage{fullpage}
\usepackage{environ}
\usepackage{xcolor}
\usepackage[thinc]{esdiff}

\def\final{0}  
\def\iflong{\iffalse}
\ifnum\final=0  
\newcommand{\mnote}[1]{{\color{red}[{Madhu: \bf #1}]\marginpar{\color{red}*}}}
\newcommand{\enote}[1]{{\color{green}[{\small Elena: \bf #1}]\marginpar{\color{red}*}}}
\newcommand{\minote}[1]{{\color{purple}[{Minshen: \bf #1}]\marginpar{\color{red}*}}}
\newcommand{\todo}[1]{{\color{red}[{ TODO: \bf #1}]\marginpar{\color{red}*}}}
\else 
\newcommand{\mnote}[1]{}
\newcommand{\enote}[1]{}
\newcommand{\minote}[1]{}
\newcommand{\todo}[1]{}
\fi

\title{{Limitations of Mean-Based Algorithms for Trace Reconstruction at Small Distance}
\thanks{A preliminary version of this work appeared in \textit{2021 IEEE International Symposium on Information Theory (ISIT)}.}}

\author{Elena Grigorescu\thanks{Purdue University, Email: \{elena-g, zhu628\}@purdue.edu. Research supported in part by NSF CCF-1910659 and NSF CCF-1910411}
\and Madhu Sudan\thanks{Harvard University, Email: madhu@cs.harvard.edu. Research supported in part by a Simons Investigator Award  and  NSF Award  CCF  1715187.}
\and Minshen Zhu\footnotemark[2]
}

\date{\today}

\usepackage{amsmath}
\usepackage{amsfonts}
\usepackage{mathtools}

\newtheorem{proposition}{Proposition}
\newtheorem{lemma}{Lemma}
\newtheorem{theorem}{Theorem}
\newtheorem{corollary}{Corollary}

\newtheorem{remark}{Remark}

\newcommand{\abs}[1]{\left\vert#1\right\vert}
\newcommand{\set}[1]{\left\{#1\right\}}
\newcommand{\tuple}[1]{\left(#1\right)} 
\newcommand{\eps}{\varepsilon}
\newcommand{\tp}{\tuple}

\def\*#1{\mathbf{#1}}
\def\+#1{\mathcal{#1}}
\def\Exp{\mathop{\mathbb E}}

\newcommand{\norm}[1]{\left\lVert#1\right\rVert}

\newcommand{\itn}[1]{^{(#1)}}

\newcommand{\N}{\mathbb{N}}

\newcommand{\bmat}[1]{\begin{bmatrix} #1 \end{bmatrix}}

\newcommand{\dH}{d_{\textsf{H}}}
\newcommand{\dE}{d_{\textsf{E}}}

\begin{document}

\maketitle

\begin{abstract}
    {\em Trace reconstruction} considers the task of recovering an unknown string $\*x\in\{0,1\}^n$ given a number of independent ``traces'', i.e., subsequences of $\*x$ obtained by randomly and independently deleting every symbol of $\*x$ with some probability $p$. The information-theoretic limit of the number of traces needed to recover a string of length $n$ is still unknown. This limit is essentially the same as the number of traces needed to determine, given strings $\*x$ and $\*y$ and traces of one of them, which string is the source.
    
    The most-studied class of algorithms for the worst-case version of the problem are ``mean-based'' algorithms. These are a restricted class of distinguishers that only use the mean value of each coordinate on the given samples. In this work we study limitations of mean-based algorithms on strings at small Hamming or edit distance.
    
    We show that, on the one hand, distinguishing strings that are nearby in Hamming distance is ``easy'' for such distinguishers. On the other hand, we show that distinguishing strings that are nearby in edit distance is ``hard'' for mean-based algorithms. Along the way, we also describe a connection to the famous Prouhet-Tarry-Escott (PTE) problem, which shows a barrier to finding explicit hard-to-distinguish strings: namely such strings would imply explicit short solutions to the PTE problem, a well-known difficult problem in number theory. Furthermore, we show that the converse is also true, thus, finding explicit solutions to the PTE problem is equivalent to the problem of finding explicit strings that are hard-to-distinguish by mean-based algorithms.
    
    Our techniques rely on complex analysis arguments that involve careful trigonometric estimates, and algebraic techniques that include applications of Descartes' rule of signs for polynomials over the reals.
\end{abstract}
\thispagestyle{empty}

\section{Introduction}
In the trace-reconstruction problem, a string $\*x\in \{0,1\}^n$ is sent over a deletion channel, which deletes each entry independently, with probability $p\in[0,1)$, resulting in a {\em trace} $\tilde{\*x}\in \{0,1\}^{\ell}$ of smaller length. The goal is to reconstruct $\*x$ exactly, from a small set of independent traces. The trace-reconstruction problem was introduced by Batu, Kannan, Khanna and McGregor \cite{BatuKKM04}, motivated by a natural problem in computational biology, in which a common ancestor DNA sequence is sought from a set of similar DNA sequences that might have resulted from the process of  random deletions in the ancestor DNA.   The information-theoretic limits and tight complexity of this problem have proven elusive so far, despite significant followup interest in a variety of relevant settings \cite{BatuKKM04, Kannan005, ViswanathanS08, HolensteinMPW08, McGregorPV14, PeresZ17, nazarov2017trace, de2019optimal, GabrysM17,  HoldenPP18, HoldenL18, HartungHP18, GabrysM19, cheraghchi2020coded, KrishnamurthyM019, brakensiek2020coded, chen2020polynomialtime, chase2021separating, narayanan2021circular}. 
The current upper bound in the worst-case formulation was recently improved by Chase \cite{chase2021separating}, who showed that $\exp(\tilde O(n^{1/5}))$ traces are sufficient for reconstruction, thus beating the previous record of $\exp(O(n^{1/3}))$ traces due to \cite{nazarov2017trace, de2019optimal}.  
However, the most general lower bound is only ${\tilde{\Omega}}(n^{3/2})$ \cite{HoldenL18, chase2021new}, hence leaving the status of the problem widely open.

To gain more insight into the trace-reconstruction problem, we study the {\em trace-distinguishing} variant, in which, given two string $\*x, \*y\in \{0,1\}^n$, the algorithm receives traces from one of the two trace distributions and is tasked to output the correct one. The trace-distinguishing problem is information theoretically equivalent to the classical trace-reconstruction problem \cite{HolensteinMPW08}. 
From a computational standpoint, the same upper and lower bounds as for the general problem hold for the trace-distinguishing variant.

In this work we aim to get more insight into the worst-case trace distinguishing problem from understanding the role of {\em distance} in the complexity of the problem. We ask the following questions: Are all pairs of  strings that are close in Hamming distance easily distinguishable? Are all pairs of strings that are close in edit distance easily distinguishable? Note that the strings used for showing the lower bounds in \cite{HoldenL18, chase2021new} 
only differ in two locations, and are indeed efficiently distinguishable (these were the strings $\*x=(01)^k101(01)^k$ and $\*y=(01)^k011(01)^k$). On the other hand, it is also reasonable to believe that trace distributions of 
strings that are very different from each other are also easily distinguishable. In fact, there exist ``codes'', namely sets of strings that are very far from each other, whose elements (codewords) lead to trace distributions that are very easily distinguishable from each other \cite{cheraghchi2020coded, brakensiek2020coded}. These codes can be constructed by efficient algorithms,  leading to some partial notion of explicitness that may be later exploited in further algorithms for the trace-reconstruction problem.


Here we approach the above questions by analyzing a restricted class of algorithms, namely mean-based. Mean-based algorithms only use the empirical mean of individual bits, and hence they operate by disregarding the actual samples, and computing only with the information given by the averages of each bit $\tilde{\*x}_i$ over the sample set $S$ of independent traces, namely $\Exp_S[\tilde{\*x}_i]$. While they appear restrictive, mean-based algorithms are in fact a very powerful class of algorithms -- for example, the upper bounds of \cite{de2019optimal, nazarov2017trace} are obtained via mean-based algorithms.

However, there exist strings $\*x, \*y \in \{0,1\}^n$ \cite{de2019optimal, nazarov2017trace} that mean-based algorithms cannot distinguish with fewer than 
 $\exp(\Omega(n^{1/3}))$ traces. This lower bound is based on a result in complex analysis \cite{borwein1997littlewood}, which only implies the existence of such strings $\*x$ and $\*y$, and not what such strings would look like structurally. In particular, we don't even have efficient algorithms for constructing such strings. 

Our main results here prove that there exist {\em explicit} strings $\*x, \*y\in \{0,1\}^n$ at edit distance only $4$,  for which every mean-based algorithm requires a {\em super-polynomial} in $n$ number of samples. By ``explicit'' strings we mean strings whose support set can be described mathematically, by algebraic equations  (say, for example, $\*x\in\{0,1\}^n$ is such that $\*x_i=1$ iff $i=2^k$, for some integer $k$).
 
On the other hand, we identify some structural properties of strings at low edit distance that yield polynomial-time mean-based trace reconstruction. In \cite{KrasikovR97, KrishnamurthyM019} the authors show that strings at small Hamming distance are efficiently distinguishable. We complement these results by observing that they are efficiently distinguishable even by mean-based algorithms.

We believe that understanding structural properties that are bottlenecks (such as explicit, hard-to-distinguish strings) for the algorithms we know of, as well as understanding structural properties that lead to fast algorithms, are necessary steps towards understanding the complexity of the trace-reconstruction problem.

We formalize our results next.

\subsection{Our results}

 We start with an observation about strings at small Hamming distance.
 
\begin{restatable}{theorem}{mainConstHamming} \label{thm:main-const-Hamming}
Let $\*x, \*y \in \set{0, 1}^n$ be two distinct strings within Hamming distance $d$ from each other. There is a mean-based algorithm that distinguishes between $\*x$ and $\*y$ with high probability using $n^{O(d)}$ traces.
\end{restatable}

The result is a slight strengthening of a recent result of  \cite{KrishnamurthyM019}, who proved exactly the same bounds for general algorithms. A weaker version was also shown in \cite{KrasikovR97, scott1997reconstructing}, where it is proved that strings at Hamming distance $2k$ have distinct $k$-decks, i.e. multisets of all $n\choose k$ subsequences of length $k$. Our contribution here is essentially to notice that the techniques of \cite{KrasikovR97, scott1997reconstructing} imply that mean-based algorithms can in fact distinguish such trace distributions (see Appendix~\ref{sec:const-Hamming} for a more detailed discussion and the complete proof).

Our main results concern the negative results at small edit distance. 

\begin{restatable}{theorem}{mainEDFour} \label{thm:main-edist-four}
Assume the deletion probability $p \in (0,1)$. There exist (explicit) strings $\*x, \*y \in \set{0, 1}^n$ within edit distance 4 of each other such that any mean-based algorithm requires $\exp\tp{\Omega(\log^2 n)}$ traces to distinguish between $\*x$ and $\*y$ with high probability.
\end{restatable}

Along the way, we also formalize a connection to the famous Prouhet-Tarry-Escott (PTE) \cite{prouhet1851memoire, dickson2013history, wright-prouhetssol} problem from number theory. The PTE problem is related to classical variants of the Waring problem and problems about minimizing the norm of cyclotomic polynomials, considered by Erd\"{o}s and Szekeres \cite{erdos1959product, BorweinIngalls}. 
Perhaps not surprisingly, our explicit solution from Theorem \ref{thm:main-edist-four} is based on products of cyclotomic polynomials. 

In the PTE problem, given an integer $k\geq 0$, one would like to find sets $A$ and $B$ of integer solutions, with $A=\{\alpha_1, \alpha_2, \ldots, \alpha_s\}$ and $B=\{\beta_1, \beta_2, \ldots, \beta_s\},$ satisfying the system $\sum_{i\in [s]} \alpha_i^j=\sum_{i\in [s]} \beta_i^j$, for all $j\in [k]$, with $\alpha_i\ne \beta_j$ for all $i, j \in [s]$. The goal is to find such solutions with {\em size}  $s$ as small as possible compared to the {\em degree} $k$.  It is easy to show that, most generally, it must be the case that $s\geq k+1$; and a pigeon-hole counting argument shows the existence of solutions with $s=O(k^2)$~\cite{wright1935tarry}. With the additional constraint that the system is not satisfied for degree $k+1$,  solutions of size $s=O(k^2 \log k)$ are known to exist \cite{hua-book}. However, all these are existential, non-constructive solutions, and the only general explicit solutions have size $s=\Theta(2^k)$ (e.g., \cite{wright-prouhetssol,BorweinIngalls}).

We note that connections between the trace-reconstruction problem and the PTE problem have been previously made.
In particular, Krasikov and Roditty \cite{KrasikovR97} noticed that pairs of strings that have the same $k$ decks yield solutions to PTE systems. 

We first show that explicit strings that are exponentially hard to distinguish by mean-based algorithms imply solutions of small size to a PTE system, as follows. This can be viewed as a deeper reason for why the negative result for mean-based algorithms in~\cite{nazarov2017trace, de2019optimal} is based on non-constructive arguments.

\begin{restatable}{theorem}{hardstringtoPTE} \label{thm:hardstring-to-PTE}
Fix any $\eps \in (0, 1/3]$. Given distinct strings $\*x, \*y \in \set{0,1}^n$ such that any mean-based algorithm requires $\exp\tp{\Omega(n^\varepsilon)}$ traces to distinguish between $\*x$ and $\*y$, the following two sets constitute a solution to the degree-$k$ PTE system
\begin{align*}
    D(\*x) = \set{i \colon x_i = 1}, \quad D(\*y) = \set{i \colon y_i = 1},
\end{align*}
with size $n = (k\log^2 k)^{1/\eps}$.
\end{restatable}

We also prove a converse of this result. However, the converse is in terms of an upper bound on the magnitude of the solutions to the PTE problem, rather than the size of the solution. 

We note that the counting argument \cite{wright1935tarry, hua-book} that shows existential results for the PTE solutions in terms of size and degree,  in fact gives solutions in which the values of the integers are bounded from above by, say, an integer $M$. When the size of the solution is  $s=\Omega(k^3)$, the proof \cite{wright1935tarry, hua-book} shows that there exist solutions where $M$ is polynomial in $s$. When the size $s$ is exponential in the degree $k$, there are constructions with $M=O(s)$ \cite{wright-prouhetssol}. 
 Hence, the size of the solution and the magnitude of the solution lead to qualitatively similar bounds in interesting ranges of the parameters.
 



\begin{restatable}{theorem}{PTEtohardstring} \label{thm:PTE-to-hardstring}
Suppose $A, B \subseteq \N$ form a solution to the degree-$k$ PTE system, and let $n \coloneqq \max A \cup B$. Define the following strings $\*x, \*y \in \set{0,1}^{n+1}$:
\begin{align*}
	\forall i \in \set{0,1,\dots,n}, \quad x_i = \begin{cases}
	0 & \textup{if $i \notin A$} \\
	1 & \textup{if $i \in A$}
	\end{cases},
	\quad y_i = \begin{cases}
	0 & \textup{if $i \notin B$} \\
	1 & \textup{if $i \in B$}
	\end{cases}.
\end{align*}
Then for any $\eps>0$, $n^{\Omega(k)}$ traces are necessary for mean-based algorithms to distinguish between $0^{\ell}\*x$ and $0^{\ell}\*y$, where $\ell=n^{3+\eps}$.
\end{restatable}
We remark that $n^{\Omega(k)} = N^{\Omega(k)}$ where $N=\ell+n=\Theta(n^{3+\eps})$ is the length of strings $0^{\ell}\*x$ and $0^{\ell}\*y$. Since $k\le n$ for any PTE solutions, the largest possible bound we could get via Theorem~\ref{thm:PTE-to-hardstring} is $N^{\Omega(N^{1/(3+\eps)})}=\exp(N^{1/(3+\eps)}\log N)$. This is consistent with the results of~\cite{de2019optimal, nazarov2017trace}, which showed that $\exp(N^{1/3})$ traces are sufficient for mean-based algorithms to distinguish between any two strings of length $N$. We also note that the hard strings obtained from general PTE solutions may have unbounded edit distance, thus they do not directly imply Theorem~\ref{thm:main-edist-four}. 

The strong connection with the PTE problem, which is believed to be a difficult problem in Number Theory, may be interpreted as evidence to the difficulty of finding explicit hard-to-distinguish strings for the trace-reconstruction problem. Such hard instances could be desirable when, for example, one wants to design instance-dependent algorithms to bypass the ``mean-based barrier''. 
We remark that a similar reduction from the problem of finding small-size explicit solutions for the PTE problem to the computational hardness of the Bounded Distance Decoding problem for Reed-Solomon codes from \cite{GandikotaGG18} revealed a similar barrier for the respective decoding problem.

PTE systems appear to be intimately connected to the trace-reconstruction problem. Indeed, the analysis of mean-based algorithms often reduces to the study Littlewood-type polynomials, namely polynomials with $\{-1, 0, 1\}$ coefficients, on the complex unit circle. This in turn often involves understanding the multiplicity of the root $1$, which is again a question tightly related to the PTE problem (see discussion in Section \ref{sec:mult-of-one}).

Finally, with the tools established in this paper, we apply the Descartes rule of signs \cite{descartes1886geometrie} to complete the  proofs of some of our results, e.g., the proof of Theorem~\ref{thm:main-const-Hamming}. As another application of this rule to larger edit distances, we also obtain the following theorem, formalized in Section \ref{sec:high-edist}.

\begin{theorem} \label{thm:informal-special-higher-edist} (Informal)
Strings $x, y \in \{0,1\}^n$ with $\dE(x, y)=d\geq 1$ and certain special block structures are distinguishable by mean-based algorithms using $n^{O(d)}$ traces. In particular, the statement holds for every pair of strings at edit distance $2$.
\end{theorem}

\paragraph{This version.}
This version of our paper includes several improvements over our previous version \cite{arxiv-previous-version}. Some of these improvements are inspired by recent work of Sima and Bruck~\cite{sima2021trace}, which appeared after the previous version of our paper was published  online. In particular, here we improve Theorem~\ref{thm:multiplicity-of-one} due to a technical lemma in \cite{sima2021trace}, whose proof we simplify further in Lemma~\ref{lem:w-around-one}. We explain the differences between the proofs in Section \ref{subsec:techniques} after Theorem \ref{thm:multiplicity-of-one}.


Other changes in this version include strengthening and simplification of theorems \ref{thm:main-const-Hamming}, \ref{thm:hardstring-to-PTE}, \ref{thm:informal-special-higher-edist} and of Lemma~\ref{lem:mean-value}. In addition, Theorem~\ref{thm:PTE-to-hardstring} is a new theorem that is essentially the converse of Theorem~\ref{thm:hardstring-to-PTE}. This new theorem answers the open questions we raised in our previous version. We suggest new open problems in Section~\ref{sec:conclusion}.

\subsection{Our techniques} \label{subsec:techniques}

\paragraph{The \cite{de2019optimal, nazarov2017trace} reduction to complex analysis}  
We recall that a mean-based algorithm only works with ``mean traces'' of a string. Formally, the mean trace of string $\*x \in \set{0,1}^n$ is a vector $\*E(\*x)=\tp{E_0(\*x), \cdots, E_{n-1}(\*x)} \in [0,1]^n$, where the $j$-th coordinate is defined as 
\begin{align*}
	E_j(\*x) = \Exp_{\tilde{\*x} \sim \+D_{\*x}}\left[ \tilde{x}_j \right].
\end{align*}
It is not hard to see (e.g.,~\cite{HolensteinMPW08}) that understanding the sample complexity for distinguishing between $\*x$ and $\*y$ by mean-based algorithms essentially amounts to understanding the $\ell_1$-distance between the mean traces.

\begin{proposition} \label{prop:l1-dist}
Given strings $\*x, \*y \in \set{0,1}^n$ with $\norm{\*E(\*x)-\*E(\*y)}_{\ell_1} = \eps$, $\Omega(1/\eps)$ traces are necessary, and $O(1/\eps^2)$ traces are sufficient for mean-based algorithms to distinguish between $\*x$ and $\*y$.
\end{proposition}

Our techniques focus on analyzing the modulus of Littlewood polynomials with $\{-1, 0, 1\}$ coefficients in certain regions of the complex plane. The reduction to complex analysis was established in \cite{de2019optimal, nazarov2017trace}.
They define the associated polynomials
$ P_{\*x}(z) = \sum_{j=0}^{n-1} E_j(\*x) \cdot z^j$
and the related polynomial  $ Q_{\*x}(p+qz) = q^{-1}P_{\*x}(z)= \sum_{k=0}^{n-1}x_k \cdot \tp{p+qz}^k$ (and hence $  Q_{\*x}(z)=\sum_{k=0}^{n-1}x_k \cdot z^k$), which is obtained from writing the $E_j$'s explicitly as 
\begin{align*}
E_j(\*x) = \Exp_{\tilde{\*x} \sim \+D_{\*x}}\left[ \tilde{x}_j \right] = \sum_{k=0}^{n-1}\Pr\left[\tilde{x}_j\textup{ comes from }x_k\right] \cdot x_k = \sum_{k=0}^{n-1}\binom{k}{j}p^{k-j}q^{j+1} \cdot x_k, 
\end{align*}
Here $p$ is the deletion probability and $q=1-p$. The reduction is summarized in the following theorem.

\begin{theorem}[\cite{de2019optimal, nazarov2017trace}] \label{thm:reduction-to-comp}
The sample complexity of mean-based algorithms for the trace-distinguishing problem for strings $\*x$ and $\*y$ is lower bounded by (up to constants) the inverse of 
\[
   \sup \set{\abs{Q_{\*x}(w) - Q_{\*y}(w)} \colon w \in \partial B\tp{p; q}}, 
\]
where  $\partial B\tp{p; q}$ denotes the circle of radius $q$ centered at $p$ in the complex plane.
\end{theorem}

For completeness, we include the details of the reduction in Appendix \ref{sec:reduction}. Note that all coefficients of $Q_{\*x}(w) - Q_{\*y}(w)$ belong to $\{-1, 0, 1\}.$ 

\paragraph{Applications of Descartes' Rule of Signs}
Here we relate the $\ell_1$-distance between the mean traces of $\*x$ and $\*y$ to the multiplicity of zero of the polynomial $Q_{\*x}-Q_{\*y}$ at 1. Specifically, we show that as long as 1 is a root with multiplicity no more than $k$, the $\ell_1$-distance is at least $n^{-O(k)}$ (assuming the deletion probability $p$ is a constant).

\begin{restatable}{theorem}{multiplicityofone} \label{thm:multiplicity-of-one}
Let $\*x, \*y \in \set{0,1}^n$ be two distinct strings. Suppose the polynomial $f(z)=Q_{\*x}(z)-Q_{\*y}(z)$ has $k$ roots at $z=1$. Then for any deletion probability $p \in (0,1)$, we have
\begin{align*}
    \norm{\*E(\*x) - \*E(\*y)}_{\ell_1} \ge \frac{q}{e}\tp{\frac{q}{n}}^k.
\end{align*}
Here $q=1-p$ is the retention probability.
\end{restatable}

Our proof of Theorem~\ref{thm:multiplicity-of-one} is inspired by the proof of Lemma 6 in~\cite{sima2021trace}. The proof presented here is arguably  simpler than the one that appeared in a preliminary version of this paper, and the bound is also improved from $n^{-O(k^2)}$ to $n^{-O(k)}$. The new idea of~\cite{sima2021trace} was to find a point in the complex plane with nice properties on a circle \emph{centered at} $1$, whereas our initial proof revolved around finding such a point on a circle \emph{touching} $1$. Their analysis uses an averaging argument, which we  simplify to an application of the Maximum Modulus Principle in our new proof (see Lemma~\ref{lem:w-around-one} in Section~\ref{sec:mult-of-one}). 

It is then desirable to upper bound the multiplicity of zero at 1 for various polynomials. Descartes' rule of sign changes provides a convenient tool to achieve this.

\begin{lemma}[\cite{descartes1886geometrie}, Theorem 36, Chapter 1, Part Five of~\cite{polya1997problems}] \label{lem:Descartes}
Let $Z(p)$ be the number of real positive roots of the real polynomial $p(x)$ (counting with multiplicity) and $C(p)$ the number of changes of sign of the sequence of its coefficients. We then have $C(p) \ge Z(p)$.
\end{lemma}

We note that prior work that we are aware of on understanding the structure of polynomials with many roots at $1$ (e.g., \cite{erdelyi2014coppersmith, erdelyi2019multiplicity}) do not appear to imply our bounds on the complex unit circle.

\paragraph{Remark} If $p(x)=a_0+a_1x+a_2x^2+ \ldots +a_nx^n$ is a polynomial, we say a pair $(i, j)$ ($0 \le i < j \le n$) is a \emph{sign change} if $a_ia_j < 0$ and $a_{i+1} = a_{i+2} = \ldots = a_{j-1} = 0$. $C(p)$ exactly counts the number of such pairs $(i,j)$.

We use this rule to prove the formal version of Theorem \ref{thm:informal-special-higher-edist}, namely  Theorem~\ref{thm:special-higher-edist}. We also use it to give a simple proof of Theorem~\ref{thm:main-const-Hamming}.

\paragraph{Complex analysis over shifted circles}
For the negative results (i.e., Theorem~\ref{thm:main-edist-four} and~\ref{thm:PTE-to-hardstring}), our strategy is to apply Theorem~\ref{thm:reduction-to-comp} and analyze the supremum of $|Q_{\*x}-Q_{\*y}|$. For Theorem~\ref{thm:PTE-to-hardstring}, we upper bound how much the modulus of an analytic function could change in a small neighbourhood of 1 by controlling its derivative. For Theorem~\ref{thm:main-edist-four}, the strongest barrier is that it is unclear how to control the edit distance. This seems even more difficult for non-constructive arguments such as the ones in~\cite{borwein1997littlewood} and~\cite{wright1935tarry}. Our construction is inspired by properties of product of cyclotomic polynomials and their relation to PTE solutions with special structures.

\subsection{Related work}

The first formulations of the problem were proposed by \cite{Levenshtein01a, Levenshtein01b}, and the precise formulation we study here were developed in \cite{BatuKKM04, HolensteinMPW08}, motivated by the connection with DNA reconstruction. DNA sequencing recently motivated the model of ``coded'' trace reconstruction \cite{cheraghchi2020coded, brakensiek2020coded}, in which the goal is to reconstruct codewords of a known code, rather than an arbitrary string. Furthermore, the worst-case trace-reconstruction problem was also studied in the memoryless replication-insertion channel \cite{cheraghchi2021mean}.

Besides the worst-case trace reconstruction discussed earlier in the introduction, a well-studied variant has been the average-case trace-reconstruction problem, studied in \cite{HolensteinMPW08, PeresZ17, McGregorPV14, HoldenPP18}. Here, the best current lower bound is $\Omega(\log^{5/2} n)$ \cite{HoldenL18, chase2021new}, and the best algorithms run in time $\exp\tp{O(\log^{1/3} n)}$. 

A recent intriguing result \cite{chen2020polynomialtime} considers the smooth variant, which is an intermediate model between the worst-case and the average-case models. In the smooth model, the initial string is obtain from an arbitrary worst-case string perturbed so that each coordinate is replaced by a uniformly random bit with some constant probability $0<\sigma<1$. In \cite{chen2020polynomialtime}, the authors show that in this case reconstruction can be done efficiently. 


Other variants consider string reconstruction from the multiset of substrings \cite{GabrysM17, GabrysM19}, population recovery variants \cite{FanCFSS19}, matrix reconstruction and parametrized algorithms \cite{KrishnamurthyM019}, and circular trace reconstruction \cite{narayanan2021circular}.



\paragraph{Recent Work.} After a preliminary version of this paper was published, Davies, R\'acz, Rashtchian and Schiffer considered a relaxed problem named \emph{approximate trace reconstruction} and provided efficient algorithms for several classes of strings~\cite{davies2021approximate}. Here the goal is to recover a string that is close to the true source string in edit distance. Soon after, \cite{chen2021near},~\cite{chase2021approximate} and~\cite{chakraborty2021approximate} showed that for random source strings an approximate solution can be found with high probability using very few traces. 
We remark that approximate solutions serve as distinguishers for pairs of strings (in the specified class) that are sufficiently far from each other in edit distance. On the other hand, for distinguishing strings that are close to each other, Sima and Bruck showed that $n^{O(d)}$ traces are also sufficient,  where $d$ is their edit distance~\cite{sima2021trace}.
Of course, the algorithm proposed in~\cite{sima2021trace} is not mean-based (otherwise it would contradict Theorem~\ref{thm:main-edist-four}). Nevertheless, the analysis, to a large extent, resembles that for mean-based algorithms. 
Indeed, one of their technical contributions is improving one of our technical lemma which relates mean-based trace reconstruction and the multiplicity of zeros of certain polynomials. 

\subsection{Organization of the paper}
In Section~\ref{sec:preliminaries} we develop the necessary notations and basic facts. In Section~\ref{sec:mult-of-one} we prove Theorem~\ref{thm:multiplicity-of-one}, which is a key factor in our analysis, and use it to prove Theorem~\ref{thm:hardstring-to-PTE}. In Appendix~\ref{sec:const-Hamming} we prove Theorem~\ref{thm:main-const-Hamming}. In Section~\ref{sec:pte-to-strings} we prove Theorem~\ref{thm:PTE-to-hardstring}. In Section~\ref{sec:edist-four} we prove Theorem~\ref{thm:concrete-edist-four}, which is a more concrete version of Theorem~\ref{thm:main-edist-four}. In Section~\ref{sec:high-edist} we prove Theorem~\ref{thm:special-higher-edist}, which is an equivalent and more concrete version of Theorem~\ref{thm:informal-special-higher-edist}. In the Appendix, we also explain how to reduce the analysis of mean-based algorithms to understanding the supremum of certain polynomials over a circle in the complex plane.



\section{Preliminaries} \label{sec:preliminaries}
Given $z \in \mathbb{C}$ and $r \in \mathbb{R}_{\ge 0}$, we write 
\begin{align*}
    B(z; r) \coloneqq \set{w \in \mathbb{C} \colon |w-z| \le r}
\end{align*}
for the disk centered at $z$ with radius $r$, and write $\partial B(z; r)$ for its boundary.

Let $p(w) = a_0 + a_1 w+ \ldots + a_n w^n$ be a polynomial where $a_j \in \mathbb{C}$. Let $A \subseteq \mathbb{C}$ be a set. We define the following norms.
\begin{align*}
\norm{p}_1 = \sum_{j=0}^{n}|a_j|, \quad \norm{p}_2 = \tp{\sum_{j=0}^{n}\abs{a_j}^2}^{1/2}, \quad \norm{p}_A = \sup_{w \in A}|p(w)|.
\end{align*}
When $A = \partial B(0; 1)$ is the complex unit circle, we also write $\norm{p}_A = \norm{p}_\infty$. These norms are connected by the following inequalities.
\begin{lemma} \label{lem:norm-ineq}
Let $p$ be a degree-$n$ polynomial with real coefficients. Then
\begin{align*}
    \frac{1}{\sqrt{n+1}} \cdot \norm{p}_1 \le \norm{p}_2 \le \norm{p}_\infty \le \norm{p}_1.
\end{align*}
\end{lemma}
\begin{proof}
The first and third inequalities are applications of Cauchy-Schwartz and the triangle inequality, respectively. The second inequality comes from the following identity
\begin{align*}
\norm{p}_2^2 = \frac{1}{2\pi}\int_{0}^{2\pi}\abs{p\tp{e^{i\theta}}}^2 d\theta,
\end{align*}
where the right-hand-side is clearly upper bounded by $\norm{p}_{\infty}^2$.
\end{proof}

We will use the following bounds for a point $z \in \partial B(p;q)$.
\begin{lemma} \label{lem:trig-bounds}
Fix $p \in (0,1)$ and $q=1-p$. Let $z =p+qe^{i\theta}$ where $\theta \in (-\pi, \pi]$. The following bounds hold.
\begin{enumerate}
\item $|z|\le 1-2pq(\theta/\pi)^2$.
\item $|z-1|\le q|\theta|$.
\item For any integer $d\ge 0$, $|z^d-1|\le dq|\theta|$.
\end{enumerate}
\end{lemma}
\begin{proof}
\textit{Item 1:} By convexity of $\sin(\cdot)$ over $[0,\pi/2]$, we have 
\begin{align*}
1-\cos x = 2\sin^2\frac{x}{2} \ge 2\tp{\tp{1-\frac{x}{\pi}}\cdot\sin 0 + \frac{x}{\pi}\cdot\sin\frac{\pi}{2}}^2 = 2\tp{\frac{x}{\pi}}^2
\end{align*} 
for $x \in [0,\pi]$. Thus we have 
\begin{align*}
|w| &= \sqrt{\tp{p+q\cos\theta}^2+\tp{q\sin\theta}^2} \\
&= \sqrt{p^2+2pq\cos\theta+q^2} \\
&= \sqrt{(p+q)^2 - 2pq(1-\cos\theta)} \\
&\le \sqrt{1 - 4pq\tp{\frac{\theta}{\pi}}^2} \\
&\le 1-2pq\tp{\frac{\theta}{\pi}}^2,
\end{align*}
where the last line is due to $(1+x)^r \le 1 + rx$ for $r \in [0,1]$ and $x \ge -1$ (Bernoulli's inequality).

\textit{Item 2:} By elementary identities for trigonometric functions, we have
\begin{align*}
\abs{e^{i\theta}-1}=\abs{\cos\theta-1+i\sin\theta}=\abs{2\sin\frac{\theta}{2}}\cdot\abs{-\sin\frac{\theta}{2}+i\cos\frac{\theta}{2}}=2\sin\frac{|\theta|}{2}.
\end{align*}
Therefore
\begin{align*}
	|z-1| &= \abs{p+qe^{i\theta}-1} = q\abs{e^{i\theta}-1} = q\cdot 2\sin\frac{|\theta|}{2} \le q|\theta|.
\end{align*}
The inequality is due to $\sin x \le x$ for $x\ge 0$.

\textit{Item 3:} Due to Item 1 and the triangle inequality, we have
\begin{align*}
	\abs{z^d-1} = |z-1| \cdot \abs{\sum_{j=0}^{d-1}z^j} \le q|\theta| \cdot \sum_{j=0}^{d-1}|z|^j \le dq|\theta|.
\end{align*}
\end{proof}

In this paper, ``with high probability'' means with probability at least $2/3$. We will use $p$ for the deletion probability and $q=1-p$. In this paper $p$ and $q$ will be constants. Given a string $\*a \in \set{0,1}^n$, a trace $\tilde{\*a} \in \set{0,1}^{\le n}$ is a subsequence of $\*a$ obtained by deleting each bit of $\*a$ independently with probability $p$. The length of $\tilde{\*a}$ is denoted by $|\tilde{\*a}|$. For $0 \le j \le n-1$, the $j$-th bit of $\*a$ and $\tilde{\*a}$ are written as $a_j$ and $\tilde{a}_j$, respectively. The distribution of $\tilde{\*a}$ is denoted by $\+D_{\*a}$. We also associate to $\*a$ the following polynomial
\begin{align*}
Q_{\*a}(w) \coloneqq a_0 + a_1w + a_2w^2 + \ldots + a_{n-1}w^{n-1}. 
\end{align*}
The degree of $Q_{\*a}$ is at most $n-1$. 

For strings $\*x, \*y \in \set{0,1}^n$, we will write $\dH(x,y)$ for the Hamming distance between $\*x$ and $\*y$,
 where $\dH(\*x,\*y)=| \{i\in [n] \colon x_i\ne y_i\}|$; 
and write $\dE(\*x, \*y)$ for the edit distance between $\*x$ and $\*y$, namely the minimum number of insertions and deletions that transform $\*x$ into $\*y$.

\section{Large $\ell_1$-distance between mean traces from low multiplicity of root 1} \label{sec:mult-of-one}
In this section we prove Theorem~\ref{thm:multiplicity-of-one}, which will be a key stepping stone to obtain our main results. We need the following lemma, which finds a point $w$ in the neighbourhood of 1 with nice properties. This lemma is  first proven in~\cite{sima2021trace} with $p\le 1/2$, and here we give a simpler proof which works for any $p\in (0,1)$.
\begin{lemma} \label{lem:w-around-one}
Let $f(z)$ be a polynomial of degree $n$. Suppose we can write 
\begin{align*}
	f(z) = (z-1)^k g(z)
\end{align*}
for some polynomial $g$ with $|g(1)| \ge 1$. Then for any $p\in(0,1)$ and $q=1-p$, there exists $w \in \mathbb{C}$ such that $|(w-p)/q|^n\le e$ and $|f(w)| \ge (q/n)^k$.
\end{lemma}
\begin{proof}
Let $\Gamma=B(1;q/n)$ denote the closed disk with radius $q/n$ centered at $1$ on the complex plane. By the Maximum Modulus Principle (see, e.g., Theorem 1.3 in Chapter III, \S 1 of~\cite{lang2013complex}), there exists a point $w \in \partial \Gamma$ such that 
\begin{align*}
	|g(w)| = \sup_{z \in \Gamma}|g(z)| \ge |g(1)| \ge 1.
\end{align*}
We denote $w_0\coloneqq w-1$. Therefore $|w_0|=q/n$, and 
\begin{align*}
	\abs{\frac{w-p}{q}}^n = \abs{\frac{1+w_0-p}{q}}^n = \abs{\frac{q+w_0}{q}}^n \le \tp{1+ \frac{|w_0|}{q}}^n = \tp{1+\frac{1}{n}}^n \le e. 
\end{align*}
Finally, we also have
\begin{align*}
	\abs{f(w)} = |w-1|^k \cdot |g(w)| \ge (q/n)^k.
\end{align*}
\end{proof}

Now we can prove Theorem~\ref{thm:multiplicity-of-one}. We recall the statement below.
\multiplicityofone*
\begin{proof}
We recall the definition of $P_{\*x}$:
\begin{align*}
    P_{\*x}(z) \coloneqq \sum_{j=0}^{n-1}E_j(\*x) \cdot z^{j}.
\end{align*}
The following identity is proven in~\cite{de2019optimal, nazarov2017trace} (see Appendix~\ref{sec:reduction} for a proof):
\begin{align}
    P_{\*x}\tp{\frac{w-p}{q}} = q\cdot Q_{\*x}(w).
    \label{eqn:comb-id}
\end{align}
Since $f(z)=Q_{\*x}(z)-Q_{\*y}(z)$ is a polynomial of degree $n$ with $k$ roots at $z=1$, we can write $f(z)=(z-1)^kg(z)$ for some polynomial $g$ such that $g(1)\neq 0$. We can also conclude that $|g(1)|\ge 1$ since $g$ has integer coefficients (to see this, consider $g(z+1)=f(z+1)/z^k$). Therefore, we can apply Lemma~\ref{lem:w-around-one} to $f$ and obtain $w$ such that $|(w-p)/q|^n\le e$ and $|f(w)|\ge(q/n)^k$. By the triangle inequality, we have
\begin{align*}
    \abs{P_{\*x}\tp{\frac{w-p}{q}}-P_{\*y}\tp{\frac{w-p}{q}}} = \abs{\sum_{j=0}^{n-1}\tp{E_j(\*x)-E_j(\*y)}\cdot \tp{\frac{w-p}{q}}^j} \le e\cdot \sum_{j=0}^{n-1}\abs{E_j(\*x)-E_j(\*y)}.
\end{align*}
On the other hand, equation (\ref{eqn:comb-id}) gives
\begin{align*}
\abs{P_{\*x}\tp{\frac{w-p}{q}}-P_{\*y}\tp{\frac{w-p}{q}}} = q\abs{Q_{\*x}(w) - Q_{\*y}(w)} \ge q\tp{\frac{q}{n}}^k.
\end{align*}
Putting everything together, we have obtained
\begin{align*}
    \sum_{j=0}^{n-1}\abs{E_j(\*x)-E_j(\*y)} \ge \frac{q}{e}\tp{\frac{q}{n}}^k.
\end{align*}
\end{proof}

\subsection{Connection to the Prouhet-Tarry-Escott problem} \label{subsec:PTE-connection}
The following is a classical statement about the PTE problem.
\begin{theorem}[e.g. \cite{BorweinIngalls}, Proposition 1] \label{thm:PTE}
Given $s, k\in \N$ and for $\alpha_i, \beta_i\in \N$, with $i\in [s]$, the following are equivalent:
\begin{itemize}
    \item $\sum_{i=1}^s \alpha_i^j=\sum_{i=1}^{s} \beta_i^j$, for $1\leq j\leq k$, and $\sum_{i=1}^s \alpha_i^{k+1}\neq\sum_{i=1}^{s} \beta_i^{k+1}$.
    \item  $\sum_{i=1}^s x^{\alpha_i} -\sum_{i=1}^s x^{\beta_i}=(x-1)^{k+1} q(x)$ where $q\in \mathbb{Z}[x]$ and $q(1)\neq 0.$ 
\end{itemize}
\end{theorem}
This connection allows us to prove Theorem~\ref{thm:hardstring-to-PTE}.

\hardstringtoPTE*
\begin{proof}
Denote by $m$ the multiplicity of root $1$ of $Q_{\*x}-Q_{\*y}$. We consider two cases.

\textit{Case 1: $m \ge k+1$.} Let $\alpha_1, \alpha_2,\dots, \alpha_s$ enumerate the set $D(\*x)$ where $s \le n$ is the cardinality of $D(\*x)$. Similarly, we also let $\beta_1, \beta_2, \dots, \beta_s$ enumerate $D(\*y)$. Note that $D(\*x)$ and $D(\*y)$ must have the same cardinality since otherwise $\*x$ and $\*y$ have different Hamming weights (and thus are distinguishable using constant traces). We have
\begin{align*}
	\sum_{i=1}^{s}x^{\alpha_i} - \sum_{i=1}^{s}x^{\beta_i} = Q_{\*x}(w) - Q_{\*y}(w) = (x-1)^{m}q(x)
\end{align*}
for some $q \in \mathbb{Z}[x]$, $q(1) \neq 0$. Therefore, Theorem~\ref{thm:PTE} implies that $D(\*x)$ and $D(\*y)$ form a solution to the degree-$(m-1)$ PTE system. In particular, they form a solution to the degree-$k$ PTE system since $m-1\ge k$.

\textit{Case 2: $m \le k$.} We will show by contradiction that this case never occurs. Otherwise, Theorem~\ref{thm:multiplicity-of-one} gives us
\begin{align*}
\sum_{j=0}^{n-1}\abs{\Exp_{\tilde{\*x}\sim\+D_{\*x}}[\tilde{x}_j] - \Exp_{\tilde{\*y}\sim\+D_{\*y}}[\tilde{y}_j]} \ge \frac{q}{e}\tp{\frac{q}{n}}^{m} \ge \frac{q}{e}\tp{\frac{q}{n}}^{k} = \exp\tp{-O(k\log n)}.
\end{align*}
On the other hand, the relation $n=(k\log^2 k)^{1/\eps}$ also gives
\begin{align*}
    n^\varepsilon = k\log^2 k, \quad \log n = \frac{1}{\varepsilon} \tp{\log k+2\log\log k} = O(\log k),
\end{align*}
which means $k\log n = O(k\log k) = o(n^\varepsilon)$ as $k, n \rightarrow \infty$. Therefore $\exp\tp{o(n^\varepsilon)}$ traces are sufficient for a mean-based algorithm to distinguish between $\*x$ and $\*y$. However, this is a contradiction to the assumption that any mean-based algorithm requires $\exp(\Omega(n^{\eps}))$ traces to distinguish between $\*x$ and $\*y$.
\end{proof}

\section{From PTE solutions to hard-to-distinguish strings} \label{sec:pte-to-strings}
In this section, we prove Theorem~\ref{thm:PTE-to-hardstring}, which says that PTE solutions imply ``hard'' strings for mean-based trace reconstruction.

The proof uses the following lemma.

\begin{lemma}[Lemma 5.4 of~\cite{borwein1999littlewood}]\label{lem:factor-coef-bound}
	Suppose 
	\begin{align*}
	p(x) = \sum_{j=0}^{n}a_j x^j, && \abs{a_j} \le 1, a_j \in \mathbb{C} \\
	p(x) = (x-1)^k q(x), \quad q(x) = \sum_{j=0}^{n-k}b_j x^j, && b_j \in \mathbb{C}.
	\end{align*}
	Then
	\begin{align*}
	\norm{q}_1=\sum_{j=0}^{n-k}\abs{b_j} \le (n+1)\tp{\frac{e n}{k}}^k.
	\end{align*}
\end{lemma}

The following lemma is an analogue of the Mean Value Theorem for analytic functions.

\begin{lemma} \label{lem:mean-value}
	Let $f(z)$ be an analytic function on an open set $D$, such that $|f'(z)| \le M$ for all $z \in D$. Then for $z_0, z$ in the closure of $D$ such that the line connecting $z$ and $z_0$ is contained in $D$, we have
	\begin{align*}
	|f(z)| \le |f(z_0)| + M\cdot \abs{z-z_0}.
	\end{align*}
\end{lemma}
\begin{proof}
	We write $f(x+yi)=u(x,y)+iv(x,y)$ for functions $u,v\colon \mathbb{R}^2\rightarrow\mathbb{R}$.
	
	Since $f(z)$ is an analytic function, it satisfies the Cauchy-Riemann equations (see, for instance, Chapter I, \S 6 of~\cite{lang2013complex}): 
	\begin{align*}
		\diffp{u}{x} = \diffp{v}{y}, \diffp{u}{y}=-\diffp{v}{x}, \textup{ and }f'(x+yi)=\diffp{u}{x} - \diffp{u}{y}i.
	\end{align*}
	Let $r(x,y)\coloneqq \abs{f(x+yi)}$. In other words, $r^2=u^2+v^2$. Taking partial derivatives of $x$ and $y$ on both sides gives
	\begin{align*}
		2r\cdot\diffp{r}{x} &= 2u\cdot\diffp{u}{x} + 2v\cdot\diffp{v}{x}, \\
		2r\cdot\diffp{r}{y} &= 2u\cdot\diffp{u}{y} + 2v\cdot\diffp{v}{y}.
	\end{align*}
	Squaring both sides of both equations and combining give
	\begin{align*} 
	\norm{\nabla r}_2 \coloneqq \norm{\tp{\diffp{r}{x}, \diffp{r}{y}}}_2 = \abs{f'(z)}\le M. 
	\end{align*}
	Now consider the auxiliary function 
	\begin{align*}
	h(t) \coloneqq \abs{f\tp{(1-t)z_0+tz}} = r\tp{x_t, y_t}
	\end{align*}
	where $t \in [0,1]$, and $x_t, y_t \in \mathbb{R}$ are such that $(1-t)z_0+tz=x_t+y_ti$. By the chain rule and Cauchy-Schwartz, for all $t \in (0,1)$ we have
	\begin{align*}
		h'(t) = \langle \nabla r(x_t, y_t), (x-x_0, y-y_0) \rangle \le \norm{\nabla r}_2 \cdot \norm{(x-x_0, y-y_0)}_2 \le M\cdot \abs{z-z_0}.
	\end{align*}
	By the Mean Value Theorem, there exists $\bar{t} \in (0,1)$ such that
	\begin{align*}
		h'(\bar{t}) = \frac{h(1)-h(0)}{1-0} = |f(z)| - |f(z_0)|.
	\end{align*}
	This implies 
	\begin{align*}
	|f(z)| = |f(z_0)| + h'(\bar{t}) \le |f(z_0)| + M\cdot \abs{z-z_0}.
	\end{align*}
\end{proof}

\begin{lemma} \label{lem:sup-subarc}
Let $f(z)$ be a polynomial of degree $n$ which can be factorized as $f(z)=(z-1)^{k+1}q(z)$ for some polynomial $q(z)$. Then for any $\alpha > 0$ we have
\begin{align*}
	\sup\set{\abs{f\tp{p+qe^{i\theta}}} \colon |\theta| < 1/\tp{qn^{1+\alpha}}} < 12n^{1-\alpha k}.
\end{align*}
\end{lemma}
\begin{proof}
Let $\theta$ be such that $|\theta|<1/(qn^{1+\alpha})$. Item 2 of Lemma~\ref{lem:trig-bounds} implies $|z-1|\le q|\theta|\le 1/n^{1+\alpha}$.

Denote $g(z)=(z-1)q(z)$. By Lemma~\ref{lem:factor-coef-bound}, we have 
\begin{align*}
\norm{g}_1 \le (n+1)\tp{\frac{en}{k}}^{k} < 6n^{k+1}.
\end{align*}
Therefore $|g'(z)| \le (n+1) \cdot \norm{g}_1 \le 12n^{k+2}$. Applying Lemma~\ref{lem:mean-value} with $D$ being the open unit disk, $z_0=1$ and $z=p+qe^{i\theta}$, we have
\begin{align*}
	\abs{g\tp{z}} \le |g(1)| + 12n^{k+2}\cdot \abs{z-1} \le 12n^{k+2} \cdot 1/n^{1+\alpha} < 12n^{k+1}.
\end{align*}
The lemma follows since
\begin{align*}
	|f(z)| = |z-1|^k \cdot |g(z)| \le \tp{1/n^{1+\alpha}}^{k} \cdot 12n^{k+1} = 12n^{1-\alpha k}.
\end{align*}
\end{proof}

Now we are ready to prove Theorem~\ref{thm:PTE-to-hardstring}. We recall the statement below.

\PTEtohardstring*

\begin{proof}[Proof of Theorem~\ref{thm:PTE-to-hardstring}]
We write $f\coloneqq Q_{\*x}-Q_{\*y}$. Due to Theorem~\ref{thm:reduction-to-comp}, it suffices to show that
\begin{align*}
	\sup\set{\abs{w^{\ell}f(w)}\colon w \in \partial B(p;q)} \le n^{-\Omega(k)}.
\end{align*}
Writing $w = p+qe^{i\theta}$ where $\theta \in (-\pi, \pi]$, we prove the theorem in the following two cases.

\textit{Case 1:} $|\theta| \ge 1/(qn^{1+\eps/3})$. 

By Item 1 of Lemma~\ref{lem:trig-bounds}, we have
\begin{align*}
|w| \le 1-2pq\tp{\frac{\theta}{\pi}}^2\le 1-2pq\tp{\frac{1}{4qn^{1+\eps/3}}}^2\le 1-\frac{p}{8qn^{2+2\eps/3}}.
\end{align*}
Therefore
\begin{align*}
	\abs{w^{\ell}f(w)} = |w|^{\ell} \cdot |f(w)| \le \tp{1-\frac{p}{8qn^{2+2\eps/3}}}^{n^{3+\eps}} \cdot (n+1) \le \exp\tp{-\Omega(pn^{1+\eps/3}/q)} < n^{-\Omega(pk/q)}.
\end{align*}
The last inequality is due to $n^{1+\eps/3}>n\ln n\ge k\ln n$ for large enough $n$.

\textit{Case 2:} $|\theta| < 1/(qn^{1+\eps/3})$.

We recall that $A$ and $B$ form a solution to the degree-$k$ PTE system. According to the definition of $\*x$, $\*y$ and Theorem~\ref{thm:PTE}, the polynomial $f$ can be factorized as $f(z)=(z-1)^{k+1}q(z)$ for some polynomial $q(z)$. Therefore, we can apply Lemma~\ref{lem:sup-subarc} with $\alpha=\eps/3$ and obtain that 
\begin{align*}
|w^{\ell}f(w)| \le |f(w)| < 12n^{1-\eps k/3}.
\end{align*}

Combining the two cases, we have 
\begin{align*}
\sup\set{\abs{w^{\ell}f(w)} \colon w\in \partial B(p;q)}\le n^{-\Omega(k)}.
\end{align*}
\end{proof}

\section{Hard strings at edit distance 4} \label{sec:edist-four}
The goal of this section is to prove Theorem~\ref{thm:main-edist-four}, and thus exhibit two strings at edit distance $4$ such that every mean-based algorithm requires super-polynomially many traces.

We will prove the following theorem, which is a more concrete version of Theorem~\ref{thm:main-edist-four}.

\begin{theorem} \label{thm:concrete-edist-four}
Let $k$ be an odd integer and $n=\sum_{j=0}^{k}3^j$ be an even integer, and $R(w)=\prod_{j=0}^{k}\tp{1-w^{3^j}}$ be a polynomial of degree $n$. Let $E_n(w) = \sum_{j=0}^{n/2}w^{2j}$. Then $Q_{\*e}(w) \coloneqq E_n(w) - R(w)$ is a 0/1-coefficient polynomial which corresponds to a string $\*e \in \set{0, 1}^{n+1}$. Moreover, any two strings $\*x$, $\*y$ of the form $\*x=\*a10\*e$ and $\*y=\*a\*e01$ satisfy
\begin{align*}
    \sup\set{\abs{Q_{\*x}(w) - Q_{\*y}(w)} \colon w \in \partial B(p;q)} \le \exp\tp{-\Omega(\log^2 n)},
\end{align*}
where $\*a$ is an arbitrary string of length $n$. Here $p, q \in (0,1)$ are constants.
\end{theorem}
\begin{proof}
$R(w)$ has the following properties: (1) The coefficients of $R$ belong to $\set{-1,0,1}$ since each monomial occurs only once in the expansion. (2) Odd-degree terms have negative signs, and even-degree terms have positive signs. It follows that $E_n(w)-R(w)$ is a polynomial with 0/1 coefficients. 

We can write 
\begin{align*}
P(w) &= Q_{\*x}(w) - Q_{\*y}(w) = w^n\tp{(w^2-1)Q_{\*e}(w) - \tp{w^{n+2}-1}} \\
&= w^n(w^2-1)\tp{Q_{\*e}(w) - E_n(w)} \\
&= w^n(1-w^2)R(w).
\end{align*}

Consider a point $w=p+qe^{i\theta}$ on the circle $\partial B(p;q)$, where $\theta \in (-\pi, \pi]$. We consider two cases.

\textit{Case 1: } $|\theta| \ge 3^{-k/4}\pi$. 

Due to Item 1 of Lemma~\ref{lem:trig-bounds}, we have
\begin{align*}
	|w| \le 1-2pq\tp{\frac{\theta}{\pi}}^2 \le 1-2pq \cdot 3^{-k/2}.
\end{align*}
Therefore
    \begin{align*}
        |P(w)| \le |w|^n \cdot 2(n+1) \le \tp{1 - 2pq\cdot 3^{-k/2}}^n \cdot 2(n+1) \le \exp\tp{-\Omega\tp{pq\sqrt{n}}}.
    \end{align*}
    The last inequality is because $1-x < e^{-x}$ and $n=\sum_{j=0}^{k}3^j > 3^{k}$.
    
\textit{Case 2:} $|\theta| < 3^{-k/4}\pi$. 

By Item 3 of Lemma~\ref{lem:trig-bounds}, we have
    \begin{align*}
        |R(w)| &= \prod_{j=0}^{k/4-1}\abs{w^{3^j}-1} \cdot \prod_{j=k/4}^{k}\abs{w^{3^j}-1} \le \prod_{j=0}^{k/4-1} 3^jq\abs{\theta} \cdot 2^{3k/4} \le \prod_{j=1}^{k/4}\tp{3^{-j}\pi} \cdot 2^{3k/4} \\
        &\le 3^{-k^2/32} \cdot (8\pi)^{k/4} = \exp\tp{-\Omega(k^2)} = \exp\tp{-\Omega\tp{\log^2 n}}.
    \end{align*}
    Therefore $|P(w)| \le 2|R(w)| \le \exp\tp{-\Omega\tp{\log^2 n}}$.
\end{proof}

The edit distance between strings $\*x$ and $\*y$ constructed in the theorem above is clearly at most 4. Thus, Theorem~\ref{thm:main-edist-four} follows via Theorem~\ref{thm:reduction-to-comp} (see Appendix~\ref{sec:reduction} for its proof).

\begin{remark}
We make several remarks on the theorem. First, the bound is essentially tight for the constructed strings, since the polynomial $Q_{\*x}-Q_{\*y}$ has $k+2=O(\log n)$ roots at 1, and Theorem~\ref{thm:multiplicity-of-one} implies that $n^{O(k)}=\exp(O(\log^2 n))$ traces are also sufficient for distinguishing between $\*x$ and $\*y$ by mean-based algorithms. Second, the theorem exhibits two strings which attain the bound in Theorem~\ref{thm:multiplicity-of-one} for $k=\Theta(\log n)$, meaning that Theorem~\ref{thm:multiplicity-of-one} generally cannot be improved (at least in the regime $k=\Theta(\log n)$). Third, by Theorem~\ref{thm:PTE} the constructed strings imply a solution to the degree-$(k+1)$ PTE system. However, the size of the solution is exponential, since the sparsity of $Q_{\*x}-Q_{\*y}$ is $\Theta(3^k)$.   
\end{remark}

\section{Higher edit distance with special structures} \label{sec:high-edist}

In this section we show a more general result about strings at higher edit distance that have a special structure which leads to easy distinguishability. At the end we also discuss some implications about the edit distance 2 and 4 cases.

We consider pairs of strings $\*x, \*y \in \set{0,1}^n$ with the following block structure:
\begin{align*}
    \*x=\*x_1\*x_2\cdots\*x_d, \quad \*y=\*y_1\*y_2\cdots\*y_d,
\end{align*}
where for each $i=1,2,\ldots,d$, $\*x_i$ and $\*y_i$ are strings of length $\ell_i > 0$. Moreover, each block $i$ falls into one of the following cases:
\begin{enumerate}
    \item $\*x_i = \*y_i$;
    \item $\*x_i = a_i \*s_i$ and $\*y_i = \*s_i b_i$ for bits $a_i, b_i \in \set{0,1}$ and string $\*s_i$;
    \item $\*x_i = \*s_i a_i$ and $\*y_i = b_i \*s_i$ for bits $a_i, b_i \in \set{0,1}$ and string $\*s_i$;
    \item $\*x_i = a_i \*s_i$ and $\*y_i = b_i \*s_i$ for distinct bits $a_i, b_i \in \set{0,1}$ and string $\*s_i$;
    \item $\*x_i = \*s_i a_i$ and $\*y_i = \*s_i b_i$ for distinct bits $a_i, b_i \in \set{0,1}$ and string $\*s_i$.
\end{enumerate}

We remark that $\*x$ and $\*y$ of the above form must be within edit distance $2d$ of each other, yet there are certainly strings at edit distance $2d$ which fail to follow this pattern (for example  $\*x=a_1\ldots a_d \*s$ and $\*y = \*s b_1\ldots b_d$ generally do not have such a block decomposition).

We also note that if $\*x$ and $\*y$ have different Hamming weights (the Hamming weight of a string is the number of 1s in it), this makes them easily distinguishable by a mean-based algorithm. This is because their traces will exhibit a difference of $q$ in expected Hamming weight. Using $O(1/q^2)$ traces, this difference will be noticeable by a mean-based algorithm. Therefore, we will focus on the more interesting case where $\*x$ and $\*y$ have the same Hamming weight.

Since $Q_{\*x}(1)$ exactly equals to the Hamming weight of $\*x$, we know that $w-1$ is a factor of the polynomial $Q_{\*x}(w)-Q_{\*y}(w)$ if $\*x$ and $\*y$ have the same Hamming weight. It is thus natural to factor $Q_{\*x}(w)-Q_{\*y}(w)=(w-1)R(w)$ for some polynomial $R$, and study the multiplicity of zeros of $R$. The special block structure described above allows us to explicitly write down the expression for $R$, and thus to study the number of sign changes in $R$. This is the main idea in proving the following theorem, which is the formal version of Theorem~\ref{thm:informal-special-higher-edist}.

\begin{theorem} \label{thm:special-higher-edist}
Let $\*x$ and $\*y$ be strings with the special structure mentioned above. Then
\begin{align*}
    \norm{\*E(\*x)-\*E(\*y)}_{\ell_1} \ge \frac{q}{e}\tp{\frac{q}{n}}^{3d}.
\end{align*}
\end{theorem}
\begin{proof}
As a warm-up, let us first consider the case where all $a_i$'s and $b_i$'s are zero. Under this assumption cases 4 and 5 never arise in the above block decomposition. We can partition $[n]$ into three sets $S_1, S_2, S_3$, each of which collecting the indices of contiguous substrings of $\*x$ (and therefore of $\*y$) of lengths $\ell_i$ corresponding to the respective case of the first three special cases above. Let $t_i = \sum_{j=1}^{i-1}\ell_j$ be the starting index of block $i$ (note that $t_1=0$ and $t_{d+1} = n$). As we are going to decompose the polynomials $Q_{\*x}(w)$ and $Q_{\*y}(w)$ using the block structure, these indices will come in handy later. 

Recall that the polynomial $Q_{\*x}(w)$ is defined as
\begin{align*}
    Q_{\*x}(w) = x_0 + x_1 w + x_2 w^2 + \ldots + x_{n-1} w^{n-1}.
\end{align*}

Given the block structure of $\*x$ and $\*y$, we can express $Q_{\*x}(w)-Q_{\*y}(w)$ in terms of the polynomials $Q_{\*x_i}(w)-Q_{\*y_i}(w)$ as 
\begin{align*}
    Q_{\*x}(w) - Q_{\*y}(w) = \sum_{i=1}^{d}w^{t_i}\tp{Q_{\*x_i}(w) - Q_{\*y_i}(w)}. 
\end{align*}
For $i \in S_2$, we have 
\begin{align*}
Q_{\*x_i}(w)-Q_{\*y_i}(w) = wQ_{\*s_i}(w) - Q_{\*s_i}(w) = (w-1)Q_{\*s_i}(w).
\end{align*}
Similarly for $i \in S_3$, we have
\begin{align*}
Q_{\*x_i}(w)-Q_{\*y_i}(w) = Q_{\*s_i}(w) - wQ_{\*s_i}(w) = (1-w)Q_{\*s_i}(w).
\end{align*}
Putting everything together, we get
\begin{align*}
    Q_{\*x}(w) - Q_{\*y}(w) &= \sum_{i \in S_2}w^{t_i}(w-1)Q_{\*s_i}(w) + \sum_{i \in S_3}w^{t_i}(1-w)Q_{\*s_i}(w) \\
    &= (w-1)\tp{\sum_{i \in S_2}w^{t_i}Q_{\*s_i}(w) - \sum_{i \in S_3}w^{t_i}Q_{\*s_i}(w) }.
\end{align*}
Towards applying Lemma~\ref{lem:Descartes}, we are going to upper bound the number of sign changes in the second term of the above expression. We note that $Q_{\*s_i}(w)$ is a polynomial with 0/1 coefficients. Each summand $w^{t_i}Q_{\*s_i}(w)$ contains a set of monomials whose degrees are in an interval $[t_i, t_{i+1})$, and all these intervals are disjoint from each other. It follows that the number of sign changes is at most $d$. The lemma also follows by Theorem~\ref{thm:multiplicity-of-one}.

Now let us turn to the case where $a_i$'s and $b_i$'s are not necessarily zero. Due to ``linearity'' of the mapping $\*x \mapsto Q_{\*x}$ (i.e. $\*x+\*y \mapsto Q_{\*x} + Q_{\*y}$), it will be helpful to write $\*x = \*x_{\emptyset} + \*x_{\Delta}$, where $\*x_{\emptyset}$ is $\*x$ but with all the $a_i$'s replaced by zero, and $\*x_{\Delta}$ contains only the $a_i$'s. Similarly write $\*y = \*y_{\emptyset} + \*y_{\Delta}$.

We recall that $\*x$ and $\*y$ have the same Hamming weight. That means the the following two sets
\begin{align*}
A=\set{i \colon a_i = 1} \textup{ and } B = \set{i \colon b_i = 1}
\end{align*}
have the same cardinality. Let $\pi \colon A \rightarrow B$ be a matching between the $A$ and $B$. For $i \in A$ let $\sigma(i)$ be the index of $a_i$ in $\*x$ and let $\tau(i)$ be the index of $b_{\pi(i)}$ in $\*y$. It follows that $\sigma(i) = t_i$ or $t_{i+1}-1$, and $\tau(i) = t_{\pi(i)}$ or $t_{\pi(i)+1}-1$, and that 
\begin{align*}
Q_{\*x_{\Delta}}(w) - Q_{\*y_{\Delta}}(w) = \sum_{i \in A}\tp{w^{\sigma(i)}-w^{\tau(i)}}.
\end{align*}
We then have
\begin{align*}
    Q_{\*x}(w) - Q_{\*y}(w) &= \tp{Q_{\*x_{\emptyset}}(w) - Q_{\*y_{\emptyset}}(w)} + \tp{Q_{\*x_{\Delta}}(w) - Q_{\*y_{\Delta}}(w)} \\
    &= \sum_{i \in S_2}w^{t_i}(w-1)Q_{\*s_i}(w) + \sum_{i \in S_3}w^{t_i}(1-w)Q_{\*s_i}(w) + \sum_{i \in A}\tp{w^{\sigma(i)}-w^{\tau(i)}} \\
    &= (w-1)\tp{\sum_{i \in S_2}w^{t_i}Q_{\*s_i}(w) - \sum_{i \in S_3}w^{t_i}Q_{\*s_i}(w) + \sum_{i \in A}J_i(w) },
\end{align*}
where each $J_i(w)$ is a polynomial of the form
\begin{align*}
    J_i(w) = \begin{cases}
    w^{\tau(i)} + w^{\tau(i)+1} + \ldots + w^{\sigma(i)-1} & \textup{if $\sigma(i) > \tau(i)$}, \\
    -w^{\sigma(i)} - w^{\sigma(i)+1} - \ldots - w^{\tau(i)-1} & \textup{if $\sigma(i) < \tau(i)$}.
    \end{cases}
\end{align*}

Let us focus on the polynomial $R(w) = R_1(w) + R_2(w)$ where
\begin{align*}
    R_1(w) \coloneqq \sum_{i \in S_2}w^{t_i}Q_{\*s_i}(w) - \sum_{i \in S_3}w^{t_i}Q_{\*s_i}(w), \quad R_2(w) \coloneqq \sum_{i \in A}J_i(w).
\end{align*}
Once more we are going to bound the number of sign changes in $R$. Fix an arbitrary $i$ and consider two degrees $k_1 \neq k_2 \in [t_i+1, t_{i+1}-1)$. From previous discussions we know that $w^{k_1}$ and $w^{k_2}$ have the same sign in $R_1$. For $R_2$, we note that for each $j \in A$, $w^{k_1}$ and $w^{k_2}$ have the same coefficients in $J_j$. This is because the coefficients of $J_j$ are identically $1$ (or $-1$) in the degree interval $[\tau(i), \sigma(i))$ (or $[\sigma(i), \tau(i))$), which either contains or is disjoint with $\set{k_1, k_2}$.  Therefore, the coefficients of $w^{k_1}$ and $w^{k_2}$ are the same in $R_2$. Finally, notice that the coefficients of $R_1$ belong to $\set{0,1,-1}$, and that the coefficients of $R_2$ are integers. Therefore $w^{k_1}$ and $w^{k_2}$ have the same sign in $R=R_1+R_2$.

Given a sign change $(i,j)$ in $R$ (cf. the remark below Lemma~\ref{lem:Descartes}), we say an index $k$ \emph{cuts} $(i,j)$ if $i \le k \le j$. The above argument shows that any sign change in $R$ must be cut by some index in the set
\begin{align*}
    C = \bigcup_{i=1}^{d}\set{t_i, t_{i+1}-1} = \set{t_1} \cup \bigcup_{i=2}^{d}\set{t_i-1, t_i} \cup \set{t_{d+1}-1}.
\end{align*}
As $t_1=0$ and $t_{d+1}-1=n-1$ each cuts at most 1 sign change, and for each $i$, $t_i-1$ and $t_i$ jointly cut at most 3 sign changes, it follows that $R$ has at most $3(d-1) + 2 = 3d-1$ sign changes. 

Now we can apply Lemma~\ref{lem:Descartes} and get that the multiplicity of zero of $Q_{\*x} - Q_{\*y}$ at 1 is at most $3d$ (1 from the factor $(w-1)$, $3d-1$ from $R(w)$). By Theorem~\ref{thm:multiplicity-of-one}, we conclude that
\begin{align*}
    \norm{\*E(\*x)-\*E(\*y)}_{\ell_1} \ge \frac{q}{e}\tp{\frac{q}{n}}^{3d}.
\end{align*}
\end{proof}

\paragraph{What happens to edit distance 2 pairs?}
Given that at edit distance 4 there are already hard strings for mean-based algorithms, this is indeed a natural question to ask. In fact, we will show that a mean-based algorithm can distinguish between $\*x$ and $\*y$ using only polynomially many traces, and this will be an application of Theorem~\ref{thm:special-higher-edist}.

\begin{corollary} \label{cor:hamming-dist-two}
Let $\*x, \*y \in \set{0, 1}^n$ be two arbitrary (distinct) strings with $\dE(\*x, \*y) = 2$. Then $n^{O(1)}$ traces are sufficient for a mean-based algorithm to distinguish between $\*x$ and $\*y$.
\end{corollary}
\begin{proof}
As before we assume $\*x$ and $\*y$ have the same Hamming weight. Due to the symmetry between 0 and 1, we may assume without loss of generality that $\*x = \*a0\*b\*c$ and $\*y = \*a\*b0\*c$ for strings $\*a, \*b, \*c$ with lengths $a, b, c$, respectively, satisfying $a+b+c+1=n$. We thus have the block decompositions $\*x=\*x_1\*x_2\*x_3$ and $\*y=\*y_1\*y_2\*y_3$ where
\begin{center}
\begin{tabular}{c c c}
    $\*x_1 = \*a$, & $\*x_2 = \*b0$, & $\*x_3 = \*c$, \\
    $\*y_1 = \*a$, & $\*y_2 = 0\*b$, & $\*y_3 = \*c$.
\end{tabular}
\end{center}
Note that this falls into the special strucutre mentioned above for $d=3$. Theorem~\ref{thm:special-higher-edist} then implies
\begin{align*}
    \sum_{j=0}^{n-1}\abs{\Exp_{\tilde{\*x}\sim\+D_{\*x}}[\tilde{x}_j] - \Exp_{\tilde{\*y}\sim\+D_{\*y}}[\tilde{y}_j]} \ge \frac{q}{e}\tp{\frac{q}{n}}^9,
\end{align*}
from which it follows that $n^{O(1)}$ traces are sufficient to distinguish between $\*x$ and $\*y$. 
\end{proof}
In fact, with a more careful analysis one can nail down the constant and show that the sample complexity is $O(n^2)$, leading to a sharp transition in sample complexity from edit distance 2 to 4.

\paragraph{Other cases for edit distance 4} We also mention that the only hard pairs at edit distance 4 have the form 
\begin{align*}
    \*x &= \*a \ a_1 \ \*b \ a_2 \ \*c \ \*d \ \*e, \\
    \*y &= \*a \ \*b \ \*c \ b_1 \ \*d \ b_2 \ \*e.
\end{align*}
The hard strings given in Theorem~\ref{thm:main-edist-four} are also in this form with $\*b = \*d = \varepsilon$ (the empty string). All other pairs not in this form will have the special structure mentioned earlier, and are thus easy.

\section{Conclusions and Open Problems} \label{sec:conclusion}

In this work we showed several results about the power and limitation of mean-based algorithms in distinguishing trace distributions of strings at small Hamming or edit distance.

Going beyond mean-based algorithms is obviously a major concern. A very natural next step is to incorporate ``multi-bit statistics'', namely the joint distribution of several bits of the traces. Indeed, the upper bound obtained in~\cite{chase2021new} is based on the joint distribution of roughly $n^{1/5}$ bits. Although this seems a much more general class of algorithms, the best bound they yield so far is still exponential. We leave as an open problem the power and limitation of algorithms based on multi-bit statistics.
\section{Acknowledgements}
We are indebted to some anonymous reviewers for pointing us to several references that we previously missed and for many useful suggestions that have been incorporated in the current writeup.
\bibliographystyle{alpha}
\bibliography{references}

\appendix
\section{Mean-based algorithms and connection to complex analysis} \label{sec:reduction}

Fix a string $\*x \in \set{0, 1}^n$. The basic idea of~\cite{de2019optimal} and~\cite{nazarov2017trace} is to consider the average number of ``1''s at index $j$ in the traces of $\*x$, i.e. the expectations $E_j(\*x) \coloneqq \Exp_{\tilde{\*x} \sim \+D_{\*x}}[\tilde{x}_j]$ for $j=0,2,\ldots,n-1$, where $\tilde{x}_j = 0$ for $j > |\tilde{\*x}| - 1$. An algorithm is said to be \emph{mean-based} if its output depends only on the statistical estimates of $E_j(\*x)$ where $j=0,1,\ldots,n-1$.

\subsection{The reduction to complex analysis}
\cite{de2019optimal} and~\cite{nazarov2017trace} showed the following bound.

\begin{theorem}[\cite{de2019optimal}, \cite{nazarov2017trace}]
For all distinct $\*x, \*y \in \set{0, 1}^n$ it is the case that  
\begin{align} \label{ineq:dos-np}
\quad \sum_{j=0}^{n-1}\abs{E_j(\*x) - E_j(\*y)} > \exp\tp{-O\tp{n^{1/3}}}.
\end{align}
\end{theorem}

This result is sufficient to imply that $\exp\tp{O(n^{1/3})}$ samples can tell the difference between $\+D_{\*x}$ and $\+D_{\*y}$ with high probability. To this end, they defined the following polynomial
\begin{align*}
    P_{\*x}(z) = \sum_{j=0}^{n-1} E_j(\*x) \cdot z^j.
\end{align*}
This makes the left-hand-side of (\ref{ineq:dos-np}) simply $\norm{P_{\*x}-P_{\*y}}_1$. By writing explicitly 
\begin{align*}
E_j(\*x) = \sum_{k=0}^{n-1}\Pr\left[\tilde{x}_j\textup{ comes from }x_k\right] \cdot x_k = \sum_{k=0}^{n-1}\binom{k}{j}p^{k-j}q^{j+1} \cdot x_k, 
\end{align*}
we have that
\begin{align*}
    P_{\*x}(z) &= \sum_{j=0}^{n-1} E_j(\*x) \cdot z^j = \sum_{j=0}^{n-1} \sum_{k=0}^{n-1}\binom{k}{j}p^{k-j}q^{j+1} \cdot x_k \cdot z^j \\
    &= q\sum_{k=0}^{n-1}x_k \sum_{j=0}^{n-1}\binom{k}{j}p^{k-j}(qz)^j \\
    &= q\sum_{k=0}^{n-1}x_k \cdot \tp{p+qz}^k \\
    &= q \cdot Q_{\*x}(p+qz).
\end{align*}
In light of Lemma~\ref{lem:norm-ineq}, one might as well bound $\norm{P_{\*x} - P_{\*y}}_\infty$. Keeping in mind that the map $z \mapsto p+qz$ shifts the complex unit circle $\partial B(0;1)$ to $\partial B\tp{p; q}$, so far we have reduced the problem to understanding the following supremum
\begin{align*}
\sup \set{\abs{Q_{\*x}(w) - Q_{\*y}(w)} \colon w \in \partial B\tp{p; q}}. 
\end{align*}
Using a result of~\cite{borwein1997littlewood}, \cite{de2019optimal} and~\cite{nazarov2017trace} proved that the above supremum is at least $\exp\tp{-O(n^{1/3})}$, which is their main technical result.

To summarize, we have the following generic lemma.
\begin{lemma} \label{lem:mean-to-complex}
Let $\*x, \*y \in \set{0,1}^n$ be two strings. Then
\begin{align*}
    \frac{1}{\sqrt{n+1}}\norm{P_{\*x} - P_{\*y}}_1 \le q \cdot \sup \set{\abs{Q_{\*x}(w) - Q_{\*y}(w)} \colon w \in \partial B\tp{p; q}} \le \norm{P_{\*x} - P_{\*y}}_1.
\end{align*}
\end{lemma}
\begin{proof}
We have that $q^{-1}P_{\*x}(z) = Q_{\*x}(w)$ where $w=p+qz$. Applying Lemma~\ref{lem:norm-ineq} to the polynomial $q^{-1}\tp{P_{\*x}-P_{\*y}}$ gives the lemma.
\end{proof}

A common bound in this paper is of the form
\begin{align*}
    \norm{P_{\*x}-P_{\*y}}_1 \ge n^{-O(d)}
\end{align*}
for some parameter $d$. A standard Chernoff-Hoeffding bound argument shows that $n^{O(d)}$ traces are sufficient for a mean-based algorithm to distinguish between $\*x$ and $\*y$. On the other hand, if for some strings $\*x$ and $\*y$ one can show
\begin{align*}
    \norm{P_{\*x}-P_{\*y}}_1 \le \eps,
\end{align*}
then it is the case that $\Omega(1/\eps)$ traces are required for distinguishing between $\*x$ and $\*y$ by mean-based algorithms. For a formal discussion about the sample complexity versus various notions of distances related to the trace problem we refer the reader to \cite{HoldenL18}.

\section{Distinguishing between strings within small Hamming distance} \label{sec:const-Hamming} 
We prove Theorem~\ref{thm:main-const-Hamming} in this section. We remark that the same result was proven in~\cite{KrishnamurthyM019}, which uses a previous result regarding reconstructing strings from their ``$k$-decks'' (i.e. the multi-set of subsequences of length $k$)\cite{KrasikovR97}. One of the results in~\cite{KrasikovR97} states that strings within Hamming distance $2k$ have different $k$-decks. Therefore when the deletion probability $p \le 1 - k/n$, the traces will have length at least $k$ in expectation and we can reconstruct the $k$-deck with high probability in $n^{O(k)}$ traces. This is exactly the argument in~\cite{KrishnamurthyM019}, but we note here that this argument does not yield a mean-based algorithm.

Theorem~\ref{thm:main-const-Hamming} states that the same task can be accomplished also by mean-based algorithms. With the machinery established in this paper, this will be an immediate consequence of Descartes' rule of sign changes (Lemma~\ref{lem:Descartes}).

\mainConstHamming*
\begin{proof}
Let $Q(w) = Q_{\*x}(w) - Q_{\*y}(w)$. We will show that the multiplicity of zero of $Q$ at 1 is at most $d$. 

We note that $Q(w)$ is a polynomial with at most $d$ non-zero terms. Therefore the number of sign changes $C(Q)$ can never exceed $d$. By Lemma~\ref{lem:Descartes}, the number of real positive roots of $Q$ is at most $d$. In particular, the multiplicity of zero of $Q$ at 1 is at most $d$. Thus by Theorem~\ref{thm:multiplicity-of-one} we have 
\begin{align*}
	\norm{\*E(\*x)-\*E(\*y)}_{\ell_1} \ge \frac{q}{e}\tp{\frac{q}{n}}^d.
\end{align*}
The sample complexity bound $n^{O(d)}$ follows from Proposition~\ref{prop:l1-dist}.
\end{proof}

\end{document}